\documentclass[11pt, reqno]{amsart}
\usepackage{amscd, amsmath, amsthm, amssymb}

\usepackage[bookmarksnumbered,colorlinks, plainpages]{hyperref}
\usepackage{stackrel}
\hypersetup{
pdftitle={Research Paper},
pdfauthor={Ali Akbar Yazdan Pour},
pdfsubject={The saturation number of monomial ideals},
pdfcreator={Ali Akbar Yazdan Pour},
colorlinks,
linkcolor=blue,
citecolor=magenta,
anchorcolor=red,
bookmarksopen,
urlcolor=red,
filecolor=red,
}
\theoremstyle{plain}
\newtheorem{Theorem}{Theorem}[section]

\newtheorem{Corollary}[Theorem]{Corollary}
\newtheorem{Proposition}[Theorem]{Proposition}

\theoremstyle{definition}

\newtheorem{Problem}{Problem}

\newtheorem{Definition}{Definition}

\theoremstyle{remark}
\newtheorem{Remark}{Remark}

\newcommand{\KK}{\mathbb{K}}
\newcommand{\PP}{\mathbb{P}}
\newcommand{\XX}{\mathbb{X}}
\newcommand{\FF}{\mathbb{F}}
\newcommand{\mm}{\mathfrak{m}}
\newcommand{\qq}{\mathfrak{q}}
\newcommand{\pp}{\mathfrak{p}}
\newcommand{\sat}{\mathrm{sat}}
\newcommand{\reg}{\mathrm{reg}}
\newcommand{\depth}{\mathrm{depth}}
\newcommand{\Min}{\mathrm{Min}}
\newcommand{\G}{\mathcal{G}}
\def\cocoa{{\hbox{\rm C\kern-.13em o\kern-.07em C\kern-.13em o\kern-.15em A}}}

\begin{document}
\author[R. Abdolmaleki]{Reza Abdolmaleki}
\email{reza.abd110@gmail.com, abdolmaleki@iasbs.ac.ir}
\address{Reza Abdolmaleki,  School of Mathematics, Institute for Research in Fundamental Sciences (IPM), P.O. Box: 19395-5746, Tehran, Iran}
\address{Department of Mathematics, Institute for Advanced Studies in Basic Sciences (IASBS), Zanjan 45137-66731, Iran}

\author[A. A. Yazdan Pour]{{Ali Akbar} {Yazdan Pour}}
\email{yazdan@iasbs.ac.ir}
\address{Department of Mathematics, Institute for Advanced Studies in Basic Sciences (IASBS), Zanjan 45137-66731, Iran}

\title{The saturation number of monomial ideals}

\begin{abstract}
Let $S=\mathbb{K}[x_1,\ldots, x_n]$  be the polynomial ring  over a field $\mathbb{K}$ and $\mathfrak{m}= (x_1, \ldots, x_n)$ be the homogeneous maximal ideal of $S$. For an ideal $I \subset S$, let $\mathrm{sat}(I)$ be the minimum number $k$ for which $I \colon \mathfrak{m}^k = I \colon \mathfrak{m}^{k+1}$. In this paper, we compute the saturation number of irreducible monomial ideals and their powers. We apply this result to  find the saturation number of the ordinary powers and symbolic powers of some families of monomial ideals in terms of the saturation number of irreducible components appearing in an irreducible decomposition of these ideals. Moreover, we give an explicit formula for the saturation number of monomial ideals in two variables.
\end{abstract}

\thanks{This research was in part supported by a grant from IPM (No.1400130019). The first author would likes to thank the Institute for Research in Fundamental Sciences (IPM), Tehran, Iran.}


\maketitle

\section{Introduction}

Let $\KK$ be a field and $S=\KK[x_1,\ldots,x_n]$  be the polynomial ring over $\KK$ in the variables $x_1,\ldots,x_n$ endowed with standard grading (i.e. $\deg(x_i)=1$). 
For ideals $I$ and $J$ in $S$ the \textit{quotient ideal} of $I$ with respect to $J$  is defined as $I\colon J=\{f \in S\colon\; fJ \subset I\}$. Moreover, the \textit{saturation of $I$ with respect to $J$} is defined as $I\colon J^{\infty}:=\bigcup_{k\geq 0}( I\colon J^{k})$.  If $\KK$ is algebraically closed, the affine variety $V(I\colon J^{\infty})$ is the Zariski closure of the difference of varieties $V(I)$ and $V(J)$, that is, $V(I\colon J^{\infty})=\overline{ (V(I)-V(J))}$ (see \cite[\S4, Theorem~10]{CLO}).

Let $\mm=(x_1, \ldots, x_n)$ be the unique graded maximal ideal of $S$ and $I$ be an ideal in $S$. The ideal 
\[ I^{\sat}:=\bigcup_{k\geq 0}( I\colon \mm^{k}) \]
is called the \textit{saturation of $I$}. If $I$ is a graded ideal, then $I^{\sat}$ is indeed the largest graded ideal of $S$ which defines the same subscheme of $\PP_{\KK}^n$ as $I$.  An ideal $I$ is called saturated, if $I= I^{\sat}$ (equivalently $I \colon \mm=I$). The ideal $I^{\sat}$  is the smallest saturated ideal containing $I$. 

The Saturation of an ideal appears in \cite{GRS} to compute the vanishing ideal of a finite set of rational points in the projective space. Let $\FF_q$ be a finite field, $I$ be a graded ideal of $\FF_q[x_1, \ldots, x_n]$ such that $(I(\PP_{\FF_q}^{n-1}) \colon I) \neq I(\PP_{\FF_q}^{n-1})$ and $\XX=V(I) \subset \PP_{\FF_q}^{n-1}$. In \cite{GRS} the authors show that 
\[I(\XX) = (I + I(\PP_{\FF_q}^{n-1}))^{\sat}.\]
Notice that $I(\PP_{\FF_q}^{n-1}) = (x_i^q x_j - x_ix_j^q \colon \; 1 \leq i < j \leq n)$ (see \cite{djm-rr}).

For an ideal $I\subset S$ we have the ascending chain 
\[ I\subseteq I\colon \mm\subseteq I\colon\mm^2\subseteq \cdots \]
of ideals in $S$. Since $S$ is a Noetherian ring, there exists an integer $k\geq 0$ such that $I\colon\mm^k=I\colon\mm^{k+1}$. In this case, $I\colon\mm^k=I\colon\mm^i$ for all $i \geq k$. Set
\[ \sat(I)=\min\{k\colon\; I\colon\mm^k=I\colon\mm^{k+1}\}. \]
The number $\sat(I)$ is called the \textit{saturation number} of $I$. Therefore, $\sat(I)$ is the minimum number of steps $k$ such that $I^{\sat}=I \colon \mm^k$. Moreover, an ideal $I$ is saturated if and only if $\sat(I) = 0$. Note that $H_{\mm}^0(S/I) = \cup_{k \geq 0} (0 \colon \mm^{k})= I^{\sat}/I$, where $H_{\mm}^i(\mbox{-})$  is the $i$th local cohomology functor with respect to $\mm$. Thus, $\sat(I) = 0$ if and only if $\depth(S/I) > 0$. In particular, any squarefree monomial ideal (strictly contained in $\mm$) is saturated. Some statements on the saturation number of graded ideals and their powers can be found in \cite{HKM}. The saturation number of $\mathbf{c}$-bounded monomial ideals and their powers is studied in \cite{AHZ}. Recall that a monomial $u=x_1^{a_1}\cdots x_n^{a_n}$  is called \textit{$\mathbf{c}$-bounded} if $a_i\leq c_i$ for all $i$, where $\mathbf{c}=(c_1,\ldots,c_n)$ is an integer vector with $c_i\geq 0$. Let $\G(I)$ be the (unique) set of  minimal monomial generators of a monomial ideal $I$. A monomial ideal $I$ is called \textit{$\mathbf{c}$-bounded stable} if for all $u\in \G(I)$ and all $i<m(u)$ for which $x_iu/x_{m(u)}$ is $\mathbf{c}$-bounded, it follows that $x_iu/x_{m(u)}\in I$. Here $m(u)$ denotes the maximal number $j$ such that $x_j |u$.

In this paper we study the saturation number of monomial ideals. To this end, we compute the saturation number of  irreducible monomial ideals and their powers (Theorem~\ref{main}). Recall that an  ideal $I$ of a ring $R$ is called {\em irreducible} if $I$ cannot be written as an intersection of two ideals of $R$ that properly contain $I$. It turns out that a monomial ideal $\qq \subset S$ is irreducible if and only if  $\qq = (x_{i_1}^{a_1}, \ldots ,x_{i_r}^{a_r})$, where $a_i >0$ for all $i$ (\cite[Proposition~6.1.16]{Vil}). It is well-known that any monomial ideal $I$ has a unique (up to permutation) irredundant decomposition $I = \qq_1\cap \cdots \cap \qq_r$ such that $\qq_i$ is an irreducible monomial ideal for $i=1,\ldots, r$ (\cite[Theorem~6.1.17]{Vil}). Let $I = \qq_1\cap \cdots \cap \qq_r$ be an irredundant irreducible decomposition of $I$, that is, none of the ideals $\qq_i$ can be omitted in this presentation. We show that $ \sat(I) \leq \max\{\sat(\qq_i) \colon\; i=1, \ldots , r\}$. Moreover, we prove that the equality holds, if $I$ is $\mm$-primary (Corollary~\ref{satreg}). Corollary~\ref{pr} and Propositions~\ref{symb}, \ref{satsat} give results on the saturation number of some symbolic powers of monomial ideals. 

Recall that a monomial ideal $I$ is called {\em  stable} if for all $u\in \G(I)$ and all $i<m(u)$, one has $x_iu/x_{m(u)}\in I$. Proposition~\ref{stable} and Proposition~\ref{ggg} provide an explicit formula to compute the saturation number of stable monomial ideals and their powers. Finally in Theorem~\ref{primsat} we give an explicit formula for the saturation number of any monomial ideal in two variables.

\section{Main results}
This section is dedicated to finding the saturation number of some families of monomial ideals. Our first main result computes the saturation number of irreducible monomial ideals. Next, we compare  the saturation number of monomial ideals, their ordinary powers and symbolic powers with the saturation number of powers of the components appearing in an irredundant irreducible decomposition of these ideals.

\begin{Theorem} \label{main}
Let $\qq=(x_1^{a_1}, \ldots ,x_n^{a_n}) \subset S$  be an irreducible monomial ideal with $a_i >0$ for $i= 1, \ldots, n$. Let $i_0 \in \{1. \ldots, n\}$ be such that $a_{i_0} =\max\{a_1, \ldots, a_n\}$ and $t_k=ka_{i_0} + \sum\limits_{i \neq i_0} a_i-n+1$ for all $k \geq 1$. Then
\begin{enumerate}
\item[(a)] $\qq^k\colon\mm^{t_k-1}=\mm$. 
\item[(b)]$\sat(\qq^k)=t_k$.
\end{enumerate}
\end{Theorem}

\begin{proof} 
Without loss of generality we may assume that $a_1 \geq \cdots \geq a_n \geq 1$. Let $u=x_1^{b_1} \cdots  x_n^{b_n}$ be a monomial in $S$.  It is clear that $u \in \qq^k$ if and only if there exist non-negative integers $c_1, \ldots , c_n$ with  $\sum_{i=1}^{n}c_i=k$ such that $a_ic_i \leq b_i$ for all $i$. In other words, $u \notin \qq^k$ if and only if for all $1 \leq i \leq n$ one has
\begin{equation} \label{m}
b_i <(k-(\alpha_1+ \cdots +\alpha_{i-1}+\alpha_{i+1}+ \cdots +\alpha_{n}))a_i, 
\end{equation}
for all non-negative integers $\alpha_1, \ldots ,\alpha_{i-1},\alpha_{i+1}, \ldots ,\alpha_{n}$ with the property that $\alpha_{j}a_j\leq b_j$ for $j=1, \ldots, i-1, i+1, \ldots, n$.

(a) Let $k \geq 1$. Obviously $x_1^{ka_1-1}x_2^{a_2-1}\cdots x_n^{a_n-1}$ is in $\mm^{t_k-1}$ but not in $\qq^k$. Thus $\qq^k\colon\mm^{t_k-1} \neq (1)$ and hence $\qq^k\colon\mm^{t_k-1}\subseteq \mm$. For the reverse inclusion, let $u=x_1^{b_1} \ldots  x_n^{b_n} \in \mm^{t_k-1}$. We show that  $x_iu \in \qq^k$  for all $1 \leq i \leq n$. Suppose that $x_lu \notin \qq^k$  for some $l$. So $x_1^{b_1} \ldots x_{l-1}^{b_{l-1}} x_l^{b_l+1}x_{l+1}^{b_{l+1}} \ldots  x_n^{b_n} \notin \qq^k$. Put $\alpha_{i}=\lfloor \frac{b_i}{a_i} \rfloor$ for $1\leq i \leq n$. Using \eqref{m}, we get 
\begin{align*}
& b_l+1<(k-(\alpha_1+ \cdots +\alpha_{l-1}+\alpha_{l+1}+ \cdots +\alpha_{n}))a_l, \text{ and }\\
& b_i <(k-(\alpha_1+ \cdots +\alpha_{i-1}+\alpha_{i+1}+ \cdots +\alpha_{n}))a_i
\end{align*}
for all $i \neq l$. Since $\sum_{i=1}^{n}b_i+1=t_k=ka_1+ \sum\limits_{i=2} ^{n}a_i-n+1$, it follows that
\begin{equation*}
ka_1+ \sum_{i=2}^{n}a_i -n=\sum_{i=1}^{n}b_i\leq ka_1 +k\sum_{i=2}^{n}a_i-(n+1)-\sum_{i=1}^{n}(a_i\sum\limits_{j=1 \atop j\neq i}\alpha_j).
\end{equation*}
Thus,
\begin{equation}\label{t}
\sum_{i=2}^{n}(k-1)a_i\geq 1+\sum_{i=1}^{n}(a_i\sum\limits_{j=1 \atop j\neq i}\alpha_j) \geq 1+\sum_{i=2}^{n}(a_i\sum_{j=1}^n\alpha_j),
\end{equation}
where the  last inequality follows from the assumption $a_1 \geq a_i$ for all $i \geq 2$. 

On the other hand, since $\alpha_ia_i+(a_i-1)=\lfloor \frac{b_i}{a_i} \rfloor a_i+(a_i-1) \geq b_i$ for all $i$, it follows that 
\begin{align*}
(k-1)a_1+\sum_{i=1}^n(a_i-1) &= ka_1+ \sum\limits_{i=2} ^{n}a_i-n=\sum_{i=1}^{n}b_i \\
& \leq  \sum_{i=1}^n\alpha_ia_i+\sum_{i=1}^n(a_i-1) \leq (\sum_{i=1}^n\alpha_i)a_1+\sum_{i=1}^n(a_i-1).
\end{align*}
Therefore, $\sum_{i=1}^n\alpha_i \geq k-1$. Now, using \eqref{t} we get $\sum_{i=2}^{n}(k-1)a_i \geq 1 +\sum_{i=2}^{n}(k-1)a_i$, a contradiction.

(b) In the case that $\qq=\mm$ and $k=1$ one can easily see that $\sat(\qq)=1=t_1$. Assume that $a_1 \geq 2$ and $k \geq 1$. We show that $\qq^k\colon\mm^{t_k-2} \subsetneq \mm $. It is clear that 
$\qq^k\colon\mm^{t_k-2} \neq (1)$ because $u=x_1^{ka_1-2}x_2^{a_2-1} \cdots  x_n^{a_n-1} \in \mm^{t_k-2} \setminus \qq^k$. Hence $\qq^k\colon\mm^{t_k-2} \subseteq \mm$. On the other hand, $x_1u\notin\qq^k$. Therefore, $\mm \nsubseteq \qq^k \colon \mm^{t_k-2}$. Thus $\qq^k\colon\mm^{t_k-2} \subsetneq \mm $. This implies that $\qq^k\colon\mm^l \neq \qq^k\colon\mm^{l+1}$ for $ l<t_k$ since $ \qq^k\colon\mm^{l}=(\qq^k\colon\mm^{l-1})\colon\mm$ for all $l\geq 1$. On the other hand,  it follows from (a) that $ \qq^k\colon\mm^{l}=S$  for all $ l\geq t_k$.  Thus, $\sat (\qq^k)=t_k$.
\end{proof}

\begin{Remark}
In Theorem~\ref{main}, if $a_i=0$ for some $i$, then $\sat (\qq)=0 $, since $\depth(S/\qq)>0$.
\end{Remark}

\begin{Corollary} \label{satreg}
Let $I=\qq_1 \cap \cdots \cap \qq_r$ be an irreducible decomposition of a monomial ideal $I$. If $t=\max\{\sat(\qq_i) \colon\; i=1, \ldots , r\}$ then
\begin{enumerate}
\item[(a)] $\sat(I) \leq t$.
\item[(b)] $\sat(I)= t$ if and only if $I$ is an $\mm$-primary monomial ideal.
\end{enumerate}
\end{Corollary}

\begin{proof}
We may assume that $I=(\cap _{i=1}^{r_0}\qq_i)\cap(\cap _{r_0+1}^{r}\qq_i)$ where  $\sqrt \qq_i \neq \mm$  for $i=1, \ldots r_0$ and $ \sqrt \qq_i =\mm$ for $i=r_0+1, \ldots r$. Let $t= \max\{\sat(\qq_i) \colon\; i=1, \ldots , r\}$. Then
\[
I\colon\mm^{t-1}=\left(\bigcap _{i=1}^{r_0}(\qq_i\colon\mm^{t-1}) \right)\cap\left(\bigcap _{r_0+1}^{r}(\qq_i\colon\mm^{t-1})\right).
\]
Note that, $\qq_i\colon\mm^{t-1}=\qq_i$ for $i=1, \ldots, r_0$, since $\qq_i$ is saturated.  On the other hand, it follows from Theorem~\ref{main}(a) that $\cap _{r_0+1}^{r}(\qq_i\colon\mm^{t-1})=\mm$  and hence, $\cap _{r_0+1}^{r}(\qq_i\colon\mm^{t})=S$.  Therefore, 
\[I\colon\mm^{t-1}=
\begin{cases} \bigcap\limits_{i=1}^{r_0}\qq_i, & \text{if } r_0>0\\
\mm, & \text{O.W.}
\end{cases}\]

(a) The above discussion shows that $I\colon\mm^t=I\colon\mm^{t+1}$. Thus, $\sat(I) \leq t$.

(b) If $I$ is not $\mm$-primary, then $r_0>0$ and the above discussion shows that $I\colon\mm^{t-1}=I\colon\mm^{t}$. Hence, $\sat(I) \leq t-1<t$. For the converse, let $s \in \{ 1, \ldots , r\}$ be such that $t=\sat (\qq_s)$. As it is shown in the proof of Theorem~\ref{main}(b), we have $\qq_s\colon\mm^{t-2} \subsetneq \mm$. Therefore, $I\colon\mm^{t-2} \subsetneq \mm$ and hence $I\colon\mm^l \neq I\colon\mm^{l+1}$ for all $l<t$. This implies that $\sat(I)\geq t$. In conjunction with (a) we get the required result.
\end{proof}

\begin{Remark}\mbox{}
\begin{itemize}
\item[(i)] 
We may have the strict inequality in Corollary~\ref{satreg}(a). For example, let $I=(x_1^2,x_2)\cap (x_1,x_2^2)\cap (x_1^3,x_2^2,x_3^2) \subset \KK[x_1,x_2,x_3]$. Then, $\max\{\sat(\qq_i) \colon\; i=1,2,3\}=5$ by Theorem~\ref{main}(b). However, one can check (by \cocoa \cite{Co}) that $\sat (I)=3$.

\item[(ii)] The only if direction of Corollary~\ref{satreg}(b) can be deduced from Corollary 3.17 and Theorem 1.1 in \cite{ib-ptg} as follows: For a graded $S$-module $M$, let $M_j$ denote the additive subgroup of $M$ consisting of homogeneous elements of degree $j$ and $\mathrm{end}(M)=\max\{j \colon \, M_j \neq 0\}$. Recall from \cite{ib-ptg} that a monomial ideal $I$ is \textit{of nested type} if for any prime ideal $\pp$ associated to $I$, there exists $i\in \{1, \ldots, n\}$ such that $\pp = (x_1, \ldots, x_i)$. Also, the number $\mathrm{end} (I^\sat/I)+1$ is called the \textit{satiety} of $I$. 

Let $I$ be an $\mm$-primary monomial ideal. Then it is clear that $I$ is of nested type and that $\mm^n \subseteq I$ for some positive integer $n$. So that $I^\sat=R$. Let $t_0=\min\{n \colon\; \mm^n \subseteq I\}$. Then
\begin{equation}\label{sat of m-primary}
\begin{split}
& I \colon \mm^{t_0-1} \subsetneq R, \text{ and} \\
& I \colon \mm^{j} =R \quad \text{for all } j\geq t_0.
\end{split}
\end{equation}
This shows that $\sat(I)=t_0$. On the other hand, \eqref{sat of m-primary} implies that 
\begin{equation}
\begin{split}
& I_{t_0-1} \subsetneq R_{t_0-1}, \text{ and}\\
& I_{j} =R_j \quad \text{for all } j\geq t_0,
\end{split}
\end{equation}
which shows that 
\begin{equation}\label{satiety is sat for m-primary}
\mathrm{end}(I^\sat/I)=\mathrm{end}(R/I)=t_0-1=\sat(I)-1.
\end{equation}
Thus the satiety of $I$ is indeed $\sat(I)$ in this case. In addition, \cite[Theorem 1.1]{ib-ptg} and \eqref{satiety is sat for m-primary} implies that $\sat(I)=\reg(I)$, where $\reg(I)$ is the Castelnuovo-Mumford regularity of $I$. Now, applying \cite[Corollary 3.17]{ib-ptg} we conclude that:
\begin{align*}
\sat(I)=\reg(I)&=\max\{\reg(\qq_i)\colon \; i=1, \ldots, r\}\\
& =\max\{\sat(\qq_i)\colon \; i=1, \ldots, r\},
\end{align*}
where $\qq_1 \cap \cdots \cap \qq_r$ is an irredundant irreducible decomposition of $I$.
\end{itemize}
\end{Remark}

\begin{Corollary}\label{sat leq primary sat}
Let $I=\qq_1 \cap \cdots \cap \qq_r$ be a primary decomposition of a monomial ideal $I$. Then,
\[  \sat(I) \leq \max\{\sat(\qq_i) \colon\; i=1, \ldots , r\}. \]
\end{Corollary}

\begin{proof}
Assume that $\qq_i=\cap _{j=1}^{s_i}\qq_{i,j}$ is an irredundant irreducible decomposition of $\qq_i$ for $ i=1, \ldots , r$. Then $I=\cap _{i=1}^{r} \cap_{j=1}^{s_i}\qq_{i,j}$ is an  irreducible decomposition of $I$. Hence,
\[  \sat(I) \leq \max\{\sat(\qq_{i,j}) \colon\; i=1, \ldots , r ,j=1, \ldots , s_i\}=\max\{\sat(\qq_i) \colon\; i=1, \ldots , r\},\]
by Corollary~\ref{satreg}(a)
\end{proof}

In the following, we apply our results to study the saturation number of two different symbolic powers of $I$ as namely $I^{(k)}$ and $I^{\{k\}}$ (see Definitions~\ref{standard symbolic power}, \ref{new symbolic power}).
\begin{Definition} \label{standard symbolic power}
Let $I$ be an ideal of a ring $R$ and let $\pp_1, \ldots \pp_r$ be the minimal primes of $I$. For any given integer $k\geq 1 $  the \textit{$k$-th symbolic power} of $I$ is denoted by $I^{(k)}$ and defined as 
\[
I^{(k)}=\bigcap _{i=1}^{r}\qq_i=\bigcap _{i=1}^{r}(I^kR_{\pp_i}\cap R),
\]
where $\qq_i$ is the $\pp_i$-primary component of $I^k$ for $i=1, \ldots, r$.
\end{Definition}

\begin{Corollary} \label{pr}
Let $I=\qq_1 \cap \cdots \cap \qq_r$ be a primary  decomposition of a monomial ideal $I$.
\begin{enumerate}
\item[(a)] $\sat(I^{(k)}) \leq \max\{\sat(\qq_i^k) \colon\; \sqrt{\qq_i} \in \Min(I)\}$.
\item[(b)] If $I$ is squarefree, then 
\[
\sat(I^{(k)})= \begin{cases}
k, & \text{if } I=\mm,\\
0, & \text{O.W.}
\end{cases}
\]
\end{enumerate}
\end{Corollary}

\begin{proof}
(a) Let $I=(\cap_{i=1}^{s}\qq_i )\cap (\cap _{i=s+1}^{r}\qq_i)$ be a primary decomposition of $I$ and $\Min(I)=\{\sqrt{\qq_1}, \ldots, \sqrt{\qq_s}\}$. By \cite[Lemma 2]{GMSVV} we have $I^{(k)}=\cap _{i=1}^{s}\qq_i ^k$. Since $I^{(k)}=\cap _{i=1}^{s}\qq_i ^k$ is a primary decomposition of $I^{(k)}$, in view of Corollary~\ref{sat leq primary sat} we get the desired result.

(b) Assume that $I$ is squarefree. If $I=\mm$, then $I^{(k)}=I^k=\mm^k$, and $\sat(I^k)=k$ obviously. Let $I \neq \mm$ and $I=\cap_{i=1}^{r}\pp_i$ be an irredundant prime decomposition of $I$. Then $\pp_i \neq \mm$ for all $i$, and hence $\depth(S/\pp_i^k)>0$. Therefore, $\sat(\pp_i^k)=0$  for all $i$. Now, the assertion follows from (a).
\end{proof}

\begin{Definition}\label{new symbolic power}
Let $I=\qq_1 \cap \cdots \cap \qq_r$ be an irredundant irreducible decomposition of the monomial ideal $I$. We define the \textit{$\{k\}$-th symbolic power} $I^{\{k\}}$ of $I$ as follows:
\[
I^{\{k\}}:=\bigcap _{i=1}^{r}\qq_i^k.
\]
\end{Definition}

\begin{Proposition}
\label{symb}
Let $I=\qq_1 \cap \cdots \cap \qq_r$ be an irredundant irreducible decomposition of a monomial ideal $I$. Then for all $k \geq 1$ we have
\begin{enumerate}
\item[(a)] $I^k \subseteq I^{\{k\}}$.
\item[(b)] $\sat(I^{\{k\}}) \leq \max\{\sat(\qq_i^k) \colon\; i=1, \ldots , r\}$.
\item[(c)] If $I$ is $\mm$-primary, then 
\[  \sat(I^{\{k\}})= \max\{\sat(\qq_i^k) \colon\; i=1, \ldots , r\}. \]
\end{enumerate}
\end{Proposition}

\begin{proof}
(a) Let $u$ be a monomial in $I^k$. Then $u=u_1u_2 \ldots u_k$ such that $u_j$ is a monomial in $I$ for $j=1, \ldots, k$. Therefore, $u_j \in \qq_i$ for $j=1, \ldots , k$ and $i=1, \ldots, r$. Thus, $u \in \qq_i^k$ for $i=1, \ldots, r$ and hence $u \in I^{\{k\}}$.

(b) Since $I^{\{k\}}=\cap _{i=1}^{r}\qq_i^k$  is a primary decomposition of $I^{\{k\}}$, the assertion follows from Corollary~\ref{sat leq primary sat}.

(c) Let $I$ be  $\mm$-primary and $I=\cap _{i=1}^{r}\qq_i$ be an irredundant irreducible decomposition of $I$. Assume that $I^{\{k\}}=\cap _{i=1}^{r} (\cap _{j=1}^{s_i}\qq_{i,j})$ where $\qq_i^k=\cap _{j=1}^{s_i}\qq_{i,j}$ is an irredundant decomposition of $\qq_i^k$ for $ i=1, \ldots , r$. Using Corollary~\ref{satreg}(b)  we get
\[  \sat(I^{\{k\}}) =\max\{\sat(\qq_{i,j}) \colon\; i=1, \ldots , r ,j=1, \ldots , s_i\}=\max\{\sat(\qq_i^k) \colon\; i=1, \ldots , r\}. \qedhere
\]
\end{proof}

\begin{Proposition} \label{satsat}
Let $I$ be an $\mm$-primary monomial ideal and $I=\qq_1 \cap \cdots \cap \qq_r$ be an irredundant irreducible decomposition of $I$. Then 
\[ \sat(I^{\{k\}})=\max\{\sat(\qq_i^k) \colon\; i=1, \ldots , r\} \leq \sat(I^k). \]
\end{Proposition}

\begin{proof}
Assume that $t_k= \sat(I^{\{k\}})=\max\{\sat(\qq_i^k) \colon\; i=1, \ldots , r\}$ and $t'_k= \sat(I^k)$. 
Let $I^k=\cap _{i=1}^{r}\qq'_i$ be an irredundant irreducible decomposition of $I^k$. It follows from Theorem~\ref{main}(a) that $I^k\colon\mm^{t'_k-1}=\cap _{i=1}^{r}(\qq'_i \colon\mm^{t'_k-1})=\mm$. So Proposition~\ref{symb}(a) implies that $\mm \subseteq I^k\colon\mm^{t'_k-1}  \subseteq I^{\{k\}}\colon\mm^{t'_k-1}$.  If $t'_k<t_k$, we get $I^{\{k\}}\colon\mm^{t'_k-1} \subsetneq \mm$ by the proof of Theorem~\ref{main}(b),  a contradiction. Thus, $t'_k\geq t_k$. This completes the proof.
\end{proof}

There are $\mm$-primary monomial ideals with $\sat(I^{\{k\}}) < \sat(I^k)$ (see Remark~\ref{useful examples}(b)). On the other hand, if $\qq$ is an $\mm$-primary irreducible monomial ideal, then $\qq^{(k)}=\qq^k= \qq^{\{k\}}$ and hence $\sat(\qq^{(k)})= \sat(\qq^k)=\sat(\qq^{\{k\}})$. This shows that $\sat(I^k)$ is the best possible bound for $\sat(I^{\{k\}})$. 

In view of Proposition~\ref{satsat}, it is natural to pose for which classes of $\mm$-primary monomial ideals the equality $\sat(I^{\{k\}})= \sat(I^k)$ holds (see Problem~\ref{3333}). In the following we show that the equality holds for $\mm$-primary stable monomial ideals. For a monomial $u \in S$ we denote by $m(u)$ the maximal number $j$ such that $x_j |u$.
\begin{Definition}
Let $I\subset S$ be a  monomial ideal.
\begin{enumerate}
\item[(a)] $I$ is called \textit{stable} if for all $u\in \G(I)$ and all $i<m(u)$, one has  $x_iu/x_{m(u)}\in I$.
\item[(b)]  $I$ is called \textit{strongly stable} if for all  $u\in \G(I)$ and all $i<j$  with $x_j|u$, it follows that $x_iu/x_{j}\in I$.
\end{enumerate}
\end{Definition}

It is clear that a strongly stable monomial ideal is stable and the product of two (strongly) stable monomial ideals is (strongly) stable. For a strongly stable ideal $I$, it is proved in \cite[Theorem 1.2]{AHZ} that
\begin{equation} \label{sat of strongly stable ideal}
\sat(I)=\max\{\ell \colon \; x_n^{\ell}|u\; \text{for some }u\in \G(I)\}. 
\end{equation}
In particular, $\sat(I^k)=k\cdot \sat(I)$, if $I$ is an $\mm$-primary strongly stable ideal. Proposition~\ref{stable} shows that \eqref{sat of strongly stable ideal} holds for stable ideals as well.

\begin{Proposition}[{Compare \cite[Corollary 2.4]{ib-ptg}}] \label{stable}
Let $I \subset S$ be a stable monomial ideal. Then
\[ \sat(I)=\max\{\ell \colon \; x_n^{\ell}|u\,\text{ for some }u\in \G(I)\}. \]
\end{Proposition}

\begin{proof}
First, we show that $I\colon\mm^t=I\colon x_n^t$ for all $t\geq 1$. It is clear that $I\colon\mm^t \subseteq I\colon x_n^t$ for all $t\geq 1$. For the reverse inclusion, let $ux_n^t \in I$ for some monomial $u \in S$, and $v=x_1^{a_1}\cdots x_n^{a_n}$ be an arbitrary monomial in $\G(\mm^t)$ where $\sum_{i=1}^{n}a_i=t$.  Since $ux_n^t \in I$ and $I$ is stable, it follows $ux_1^{a_1}\cdots x_n^{a_n}=uv \in I$. So $I\colon x_n^t \subseteq I\colon\mm^t$ and hence, $I\colon\mm^t=I\colon x_n^t$. Now, let $\G(I)=\{ u_1, \ldots , u_m \}$ be the set of minimal monomial generators of $I$. It follows from \cite[Proposition~1.14]{EH} that
\begin{equation} \label{stablestable}
I\colon\mm^k=I\colon x_n^k=(u_i/\gcd(u_i, x_n^k)  \colon\; i=1, \ldots , m),
\end{equation}
for all $k \geq 1$. Let $s=\max\{\ell \colon \; x_n^{\ell}|u\  \text{ for some }u\in \G(I)\}$. Then  $x_n$ does not divide $w$ for all $w \in \G(I\colon\mm^k)$ and all $k\geq s$. Therefore, $I\colon\mm^k$ is a saturated ideal,  since $\depth(S/(I\colon\mm^k)) >0$. Thus, $\sat(I) \leq s$. On the other hand, it follows from \eqref{stablestable} that $x_n | u$ for some $u \in \G(I\colon\mm^{s-1})$. Therefore, $I\colon\mm^{s-1} \neq I\colon\mm^{s}$ and hence, $\sat (I)=s$.
\end{proof}

Let $B(u_1, \ldots,u_n)$ denote the smallest stable ideal containing monomials $u_1, \ldots, u_m$, and $\mathcal{B}(u_1, \ldots,u_n)$ be the smallest strongly  stable ideal containing $u_1, \ldots, u_m$.

\begin{Proposition}\label{ggg}
Let $I \subset S$ be an $\mm$-primary stable monomial ideal such that $x_n^{d} \in \G(I)$ for some positive integer $d$, and  $I=\cap _{i=1}^{r}\qq_i$ be an irredundant irreducible decomposition of  $I$. Then for all $k\geq 1$,
\begin{enumerate}
\item[(a)] $I^k\colon\mm^{kd-1}=\mm$,
\item[(b)] $\sat(I^k)=k\cdot \sat (I)=kd$,  
\item[(c)] $\sat (I^k)=\max\{\sat(\qq_i^k) \colon\; i=1, \ldots , r\}=\sat(I^{\{k\}}).$
\end{enumerate}
\end{Proposition}

\begin{proof}
(a) First we show that $I\colon\mm^{d-1}=\mm$. It is clear that  $I\colon\mm^{d-1} \subseteq \mm$.  For the reverse inclusion we must show that $\mm \cdot \mm^{d-1} \subseteq I$. Since $I$ is stable and $x_n^{d} \in \G(I)$, we get  $B(x_n^{d}) \subseteq I$. Therefore,
\[ \mm \cdot \mm^{d-1}=\mm^{d}=B(x_n^{d}) \subseteq I.  \]
Thus, $I\colon\mm^{d-1}=\mm$. 
Since $I^k$ is stable and $x_n^{kd} \in \G(I^k)$ the above discussion shows that $I^k\colon\mm^{kd-1}=\mm$.

(b)  Since $x_n^{kd} \in \G(I^k)$ for all $k\geq 1$, the assertion follows from Proposition~\ref{stable}.

(c)  Let $t_k=\max\{\sat(\qq_i^k) \colon\; i=1, \ldots , r\}$ for $ k\geq 1$. Since $I$ is $\mm$-primary we have $x_1^{a_1}, \ldots, x_n^{a_n} \in \G(I)$ for some positive integers $a_1, \ldots, a_n$, and since $I$ is stable, we get $a_n \geq a_{n-1} \geq \cdots \geq a_1$. It is obvious that  the irreducible ideal $(x_1^{b_1},x_2^{b_2}, \ldots,  x_{n-1}^{b_{n-1}},x_n^{a_n})$ appears in the irredundant irreducible decomposition of $I$ for some positive integers  $b_1, b_2, \ldots, b_{n-1} \leq a_n$. We denote this component by $\qq_\textbf{b}$. By Theorem~\ref{main}(b), $\sat(\qq_\textbf{b}^k)=ka_n+\sum_{i=1}^{n-1}b_i-(n-1)$ for $k\geq 1$. The ideal $\qq_\textbf{b}$  is an $\mm$-primary ideal, since $I$ is $\mm$-primary. Thus $b_j \geq 1$ for $j=1, \ldots , n-1$ and hence, $\sum_{i=1}^{n-1}b_i\geq n-1$. Therefore, $\sat(\qq_\textbf{b}^k) \geq ka_n$. Since $t_k \leq \sat (I^k)$ by Proposition~\ref{satsat}, we get $\sat(\qq_\textbf{b}^k) \leq \sat (I^k)=ka_n$ and so $\sat(\qq_\textbf{b}^k) = ka_n$.  Hence, $\sat (I^k)=t_k $ for all $k\geq 1$.
\end{proof}

\begin{Remark} \mbox{} \label{useful examples}
\begin{enumerate}
\item[(a)] The statement of Proposition~\ref{ggg}(b) may fail if $I$ is not $\mm$-primary, even when $I$ is strongly stable monomial ideal. For instance, let $I$ be the strongly stable ideal
\[ \mathcal{B}(x_1^2, x_2^2x_3^2, x_1x_2x_3x_4) \subset \KK[x_1,x_2,x_3,x_4].\]
Then,
\begin{align*}
I=(x_1^2, x_1x_2x_3^2, x_1x_2^2x_3, x_1x_2^3, x_2^3x_3, x_2^4, x_2^2x_3^2,  x_1x_2^2x_4,x_1x_2x_3x_4),
\end{align*}
and so
\begin{align*}
I^2=(& x_1^4, x_1^3x_2^2x_4, x_1^3x_2x_3x_4,  x_1^3x_2x_3^2, x_1^3x_2^2x_3, x_1^3x_2^3, x_1^2x_2^4, 
x_1^2x_2^3x_3, \\ & x_1^2x_2^2x_3^2, x_1x_2^6x_4,x_1x_2^7, x_2^8, x_1x_2^5x_3x_4, x_1x_2^6x_3, x_2^7x_3,
x_1x_2^4x_3^2x_4, \\ & x_1x_2^3x_3^3x_4, x_1x_2^3x_3^4, x_1x_2^4x_3^3, x_1x_2^5x_3^2, x_2^6x_3^2, x_2^5x_3^3, x_2^4x_3^4). 
\end{align*}
One may check that $\sat(I^2)=1=\sat(I)\neq 2 \cdot\sat(I)$. Note that, in this example $I$ is not equigenerated. If  $I$ is an equigenerated strongly stable monomial ideal, then the equality $\sat(I^k)=k\cdot \sat(I)$ follows from \cite[Corollary~1.3]{AHZ}.
 
\item[(b)] The equality of  Proposition~\ref{ggg}(c) may not hold if $I$ is not stable. In other words, the inequality $\max\{\sat(\qq_i^k) \colon\; i=1, \ldots , r\}   \leq \sat(I^k)$ in Proposition~\ref{satsat} may be strict. For example, let
\begin{equation*}
I=(x_1^{50}, x_1^{40}x_2^{10}, x_1^{39}x_2^{34}, x_1^{38}x_2^{35}, x_1^{37}x_2^{36}, x_1^{36}x_2^{37}, x_1^{35}x_2^{38},x_1^{34}x_2^{39}, x_1^{10}x_2^{40}, x_2^{50}).
\end{equation*}
Then,
\begin{equation*}
I^2=(x_1^{100}, x_1^{90}x_2^{10}, x_1^{80}x_2^{20}, x_1^{60}x_2^{40}, x_1^{50}x_2^{50}, x_1^{40}x_2^{60},x_1^{20}x_2^{80}, x_1^{10}x_2^{90}, x_2^{100}).
\end{equation*}
Using equation \eqref{pri} and Theorem~\ref{main} we get $\max\{\sat(\qq_i^2) \colon\; i=1, \ldots , r\}=113$, while using Theorem~\ref{primsat} we get $\sat(I^2)=119$.
\end{enumerate}
\end{Remark}

Our last result is devoted to finding the saturation number of monomial ideals in two variables. Let $I \subset S=\KK[x_1,x_2]$ be a monomial ideal and $\G(I)=\{x_1^{a_1}x_2^{b_1}, \ldots , x_1^{a_m}x_2^{b_m} \}$ be  the minimal set of generators of $I$. We may assume that $a_1> a_2  > \cdots > a_m \geq 0$. Then $0 \leq b_1 < b_2  < \cdots < b_m$. The following theorem gives the saturation number of $I$ in terms of the $a_i$'s and $b_i$'s.

\begin{Theorem} \label{primsat}
Let $I \subset S=\KK[x_1,x_2]$ be a monomial ideal with the minimal set of generators $\G(I)=\{x_1^{a_i}x_2^{b_i}  \}_{i= 1, \ldots , m}$. Assume that $a_1> a_2  > \cdots > a_m \geq 0$ and $0 \leq b_1 < b_2  < \cdots < b_m$.  Then 
\[ \sat(I)=s-a_m-b_1-1, \]
where $s=\max \{a_i+b_{i+1}\colon\; i=1, \ldots , m-1\}$.
\end{Theorem}

\begin{proof}
By \cite[Proposition 3.2]{MS}  $I$ has an irredundant irreducible decomposition
\begin{equation} \label{pri}
I=(x_2^{b_1}) \cap (x_1^{a_1}, x_2^{b_2}) \cap  \cdots \cap (x_1^{a_{m-1}}, x_2^{b_m}) \cap (x_1^{a_m}),
\end{equation}
where the first or last components are to be deleted if $b_1 = 0$ or $a_m = 0$. So, for all $k \geq 1$ we get
\begin{align*}
I\colon\mm^k&=\left( (x_2^{b_1})\colon\mm^k  \right) \cap \left( (x_1^{a_1}, x_2^{b_2})\colon\mm^k \right) \cap  \cdots \cap \left( (x_1^{a_{m-1}}, x_2^{b_m})\colon\mm^k \right) \cap \left( (x_1^{a_m})\colon\mm^k\right) 
\\&=(x_1^{a_m}x_2^{b_1})\cap \left(\bigcap_{i=1}^{m-1}\left((x_1^{a_i},x_2^{b_{i+1}})\colon\mm^k\right) \right).
\end{align*}

Let $s'=s-a_m-b_1-1$. Then \cite[Theorem 1.4]{V} and Theorem~\ref{main} imply that:
\begin{equation*}
(x_1^{a_i}, x_2^{b_{i+1}})\colon\mm^k = \begin{cases}
(x_1^{a_i}, x_2^{b_{i+1}})+\mm^{a_i+b_{i+1}-k-1}, & \text{if } k< a_i+b_{i+1}-1\\
S, & \text{O.W.}
\end{cases}
\end{equation*}
for all $i=1, \ldots, m-1$. Notice that $a_i+b_{i+1}-k-1 \leq a_m+b_1$ if $k< a_i+b_{i+1}-1$. So in this case  $(x_1^{a_m}x_2^{b_1}) \subset \mm^{a_i+b_{i+1}-k-1}$. Therefore, $I\colon\mm^k=(x_1^{a_m}x_2^{b_1})$ for all $k \geq s'$.

Now, it is enough to show that $I\colon\mm^{s'-1}\neq (x_1^{a_m}x_2^{b_1})$. To this end, we observe that 
\[ I\colon\mm^{s'-1}=(x_1^{a_m}x_2^{b_1})\cap \left( \bigcap _{i=1}^{m-1}\left( (x_1^{a_i},x_2^{b_{i+1}})\colon\mm^{s'-1} \right)\right). \]
Let $i_0 \in \{1, \ldots, m-1\}$ be such that
\[ \max \{a_i+b_{i+1}\colon\; i=1, \ldots , m-1\}=a_{i_{0}}+b_{i_{0}+1}. \]
Then $(x_1^{a_{i_{0}}}, x_2^{b_{i_{0}+1}})\colon\mm^{s'-1}=(x_1^{a_{i_{0}}}, x_2^{b_{i_{0}+1}})+\mm^{a_m+b_1+1}$ by \cite[Theorem 1.4]{V}. Thus,  $x_1^{a_m}x_2^{b_1}\notin (x_1^{a_{i_{0}}}, x_2^{b_{i_{0}+1}})\colon\mm^{s'-1}$. The proof is complete.
\end{proof}

\section{Problems}
In this section we propound some problems dealing with the equality $\sat(I^k)=\sat(I^{\{k\}})$ where $I$ is an $\mm$-primary monomial ideal. We observed in Proposition~\ref{satsat} that $\sat(I^{\{k\}}) \leq \sat(I^k)$, if $I$ is an $\mm$-primary monomial ideal. Meanwhile we have seen in Remark~\ref{useful examples}(b) that this inequality may be strict. Note that the ideal in Remark~\ref{useful examples}(b) is not equigenerated. Our running examples in $\KK[x,y]$ show that the equality holds in the case of $\mm$-primary equigenerated monomial ideals. Therefore, our first question is stated as follows:

\begin{Problem}\label{333}
Let $I$ be an $\mm$-primary equigenerated monomial ideal. Does the following equality hold?
\[ \sat(I^k)=\sat(I^{\{k\}}). \]
\end{Problem}

Any positive answer to one of the following statements, leads to an affirmative response to Problem~\ref{333}.
\begin{Problem}\label{111}
Let $I$ be an $\mm$-primary equigenerated monomial ideal and $I=\cap _{i=1}^{r}\qq_i$ be an irredundant irreducible decomposition of $I$. 
\begin{itemize}
\item[(a)] Let $t_k=\max\{\sat(\qq_i^k) \colon\; i=1, \ldots, r\}$ for all $k=1, \ldots, r$. Is it true that $\mm \subseteq I^k\colon\mm^{t_k}$?
\item[(b)] Let $k$ be a positive integer. Does there exist a positive integer $s_k$ such that  $s_k <\max\{\sat(\qq_i^k) \colon\; i=1, \ldots , r\}$ and $I^k\colon\mm^{l} =I^{\{k\}}\colon\mm^{l}$ for all $l \geq s_k$.
\end{itemize}
\end{Problem}

Proposition~\ref{ggg} provides a sufficient condition for an $\mm$-primary monomial ideal  satisfying the equality $\sat(I^k)= \sat(I^{\{k\}})$, namely being stable. Accordingly, we pose the following more general problem:
\begin{Problem}\label{3333}
Let $I$ be an $\mm$-primary monomial ideal. Under which conditions on $I$ does the equality $\sat(I^k)=\sat(I^{\{k\}})$ hold for all $k$?
\end{Problem}


\end{document}